\documentclass[a4paper,12pt]{article}

\usepackage{color}
\usepackage[pdftex]{graphicx}
\usepackage{amsmath,amssymb,amsthm}
\usepackage{wrapfig}
\usepackage{ifpdf}

\newtheorem{theorem}{Theorem}[section]
\newtheorem{cor}{Corollary}[section]

\newcommand{\md}{\mathop{\rm (mod}\nolimits}
\newcommand{\conv}{\mathop{\rm conv}\nolimits}

\newcommand{\cov}{\mathop{\rm cov}\nolimits}

\newcommand{\fc}{\mathop{\rm fc}\nolimits}
\newcommand{\but}{\mathop{\rm BUT}\nolimits}
\DeclareMathOperator{\dg2}{deg_2}

\title {Extensions of Sperner and Tucker's lemma for manifolds}

\author {Oleg R. Musin\thanks{This research is partially supported by NSF grant DMS - 1101688.}}

\begin{document}

	\ifpdf \DeclareGraphicsExtensions{.pdf, .jpg, .tif, .mps} \else
	\DeclareGraphicsExtensions{.eps, .jpg, .mps} \fi	
	
\date{}
\maketitle

\begin{abstract} The Sperner and Tucker lemmas are combinatorial analogous of the Brouwer and Borsuk - Ulam theorems with many useful applications. These classic lemmas are concerning labellings of triangulated discs and spheres. In this paper we show that discs and spheres can be substituted by large classes of manifolds with or without boundary. 
\end{abstract}

\medskip


\section{Introduction}

Throughout this paper the symbol ${\mathbb R}^d$ denotes the Euclidean space of dimension $d$.  We denote by  ${\mathbb B}^d$ the $d$-dimensional ball and by  ${\mathbb S}^d$ the $d$-dimensional sphere. If we consider ${\mathbb S}^d$ as the set of unit vectors $x$ in ${\mathbb R}^{d+1}$, then points $x$ and $-x$ are called {\it antipodal} and the symmetry given by the mapping  
 $x \to -x$ is called the {\it antipodality} on  ${\mathbb S}^d$. 

\subsection{Sperner's lemma}

{\it Sperner's lemma} is   a statement about labellings of triangulated simplices ($d$-balls). It is a discrete analog of the Brouwer fixed point theorem.

\medskip

\begin{figure}[h]
\begin{center}
  \includegraphics[clip]{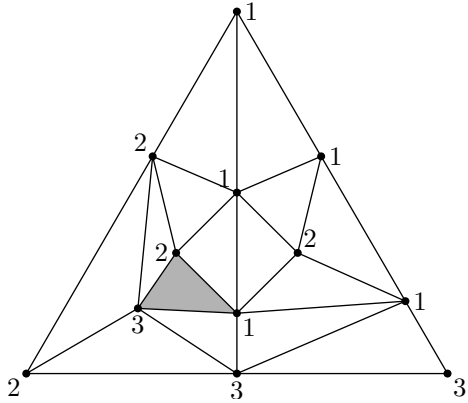}
\end{center}
\caption{A 2-dimensional illustration of Sperner's lemma}
\end{figure}

\medskip

Let $S$ be a $d$-dimensional simplex  with vertices $v_1,\ldots,v_{d+1}$. Let $T$ be a triangulation of $S$. Suppose that each vertex of $T$ is assigned a unique label from the set $\{1,2,\ldots,d+1\}$. A labelling $L$ is called {\it Sperner's} if  the vertices are labelled in such a way that a vertex   of $T$ belonging to the interior of a face $F$ of $S$ can only be labelled by $k$ if $v_k$ is on $S$.   


\begin{theorem}{\bf (Sperner's lemma \cite{Sperner})} 
	Every Sperner labelling of a triangulation of a $d$-dimensional simplex contains a cell labelled with a complete set of labels: $\{1,2,\ldots, d+1\}$.
\end{theorem}


\medskip

There are several extensions of this lemma. One of the most interesting is the De Loera - Petersen - Su theorem. In the paper \cite{DeLPS} they proved the Atanassov conjecture \cite{Atan}.
\begin{theorem} {\bf (Polytopal Sperner's lemma \cite{DeLPS})}
Let $P$ be a  polytope in ${\mathbb R}^d$  with vertices $v_1,\ldots,v_n$. Let $T$ be a triangulation of  $P$. Let $L:V(T)\to\{1,2,\ldots,n\}$ be a Sperner labelling.
Then there are at least $(n-d)$ fully-colored (i.e. with distinct labels) $d$-simplices of $T$.
\end{theorem}

\medskip

Meunier \cite{Meun} extended this theorem:
\begin{theorem}
 Let $P^d$ be  a $d$-dimensional PL manifold embedded in ${\mathbb R}^d$  that has bondary $B$. Suppose $B$ has $n$ vertices $v_1,\ldots,v_n$. Let $T$ be a triangulation of $P$.  Let $L:V(T)\to\{1,2,\ldots,n\}$ be a Sperner labelling. Let $d_i$ denote  the number of edges of $B$ which are connected to $v_i$. Then there are at least $n+\lceil\min_i\{d_i\}/d\rceil - d -1$ fully-labelled $d$-simplices such that any pair of these fully-labelled simplices receives two different labellings. 
\end{theorem}

\medskip



\subsection{Tucker's lemma}

\begin{figure}
\begin{center}

  \includegraphics[clip,scale=0.9]{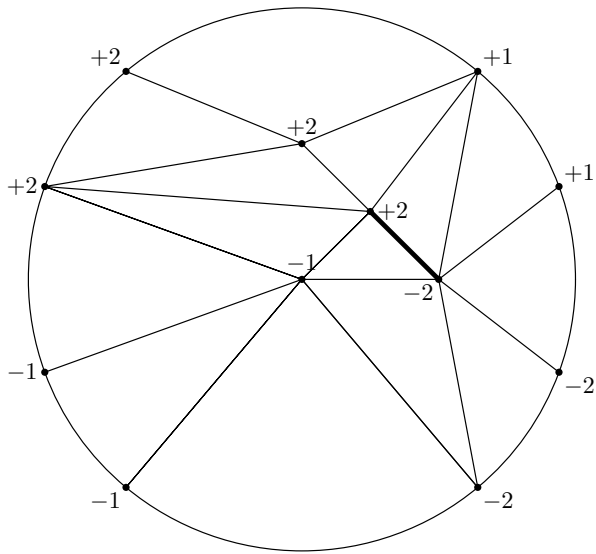}
\end{center}
\caption{A 2-dimensional illustration of Tucker's lemma}
\end{figure}

\medskip

Let $T$ be some triangulation of the $d$-dimensional ball ${\mathbb B}^d$. We call $T$ {\it antipodally symmetric on the boundary}  if the set of simplices of $T$ contained in the boundary of  ${\mathbb B}^d = {\mathbb S}^{d-1}$ is an antipodally symmetric triangulation of  ${\mathbb S}^{d-1}$, that is if $s\subset {\mathbb S}^{d-1}$ is a simplex of $T$, then $-s$ is also a simplex of $T$. 

\begin{theorem} {\bf (Tucker's lemma \cite{Tucker})} Let $T$ be a triangulation of  ${\mathbb B}^d$ that antipodally symmetric on the boundary. Let $$L:V(T)\to \{+1,-1,+2,-2,\ldots, +d,-d\}$$ be a labelling of the vertices of $T$ that satisfies $L(-v)=-L(v)$ for every vertex $v$ on the boundary. Then there exists an edge in $T$ that is {\bf complementary}, i.e. its two vertices are labelled by opposite numbers. 
\end{theorem}


Consider also the following version of Tucker's lemma:
\begin{theorem}
Let $T$ be a centrally symmetric triangulation of the sphere ${\Bbb S}^d$. Let $$L:V(T)\to \{+1,-1,+2,-2,\ldots, +d,-d\}$$ be an equivariant (or Tucker's) labelling, i.e.   $L(-v)=-L(v)$). Then there exists a complementary edge.
\end{theorem}

\medskip

Tucker's lemma was extended by Ky Fan \cite{KyFan}:
\begin{theorem}
 Let $T$ be a centrally symmetric triangulation of the sphere ${\Bbb S}^d$. Suppose that each vertex $v$ of $T$ is assigned a label $L(v)$ from $\{\pm1,\pm2,\ldots,\pm n\}$ in such a way that $L(-v)=-L(v)$. Suppose this labelling does not have complementary edges. Then there are an odd number of $d$-simplices of $T$ whose labels are of the form $\{k_0,-k_1,k_2,\ldots,(-1)^dk_d\}$, where $1\le k_0<k_1<\ldots<k_d\le n$. In particular, $n\ge d+1$.
\end{theorem}

In this paper we consider extensions of the Sperner, De Loera - Petersen - Su, Tucker and Fan theorems for manifolds. We show that for all cases $d$-balls and spheres can be substituted by $d$-manifolds with or without boundary. 
  

\section{Preliminaries}

Throughout this paper we consider manifolds that admit triangulations. The class of such manifolds  is called  {piecewise linear (PL)  manifolds.} Note that  a smooth manifold can be triangulated, therefore it is also a PL manifold. However, there are topological manifolds  that do not admit a triangulation. 

A {\em topological manifold} is a topological space that resembles Euclidean space near each point. More precisely, each point of a $d$-dimensional manifold has a neighbourhood that is homeomorphic to the Euclidean space of dimension $d$. A compact manifold without boundary is called {\em closed}. If a manifold contains its own boundary, it is called a {\it manifold with boundary.}

{\it Smooth manifolds} (also called {\it differentiable manifolds}) are manifolds for which overlapping charts ``relate smoothly'' to each other, meaning that the inverse of one followed by the other is an infinitely differentiable map from Euclidean space to itself.

$M$ is called a {\it piecewise linear (PL) manifold} if it is a topological manifold together with a piecewise linear structure on it. 
Every PL manifold $M$ admits a {\it triangulation:} that is, we can find a collection of simplices $T$ of dimensions $0, 1,\ldots, d$, such that (1) any face of a simplex belonging to $T$ also belongs to $T$, (2) any nonempty intersection of any two simplices of $T$ is a face of each, and (3) the union of the simplices of $T$ is $M$. (See details in \cite{Bryant}.) 

Note that the circle is the only one-dimensional closed manifold.  Closed manifolds in two dimensions are completely classified. (See details and proofs in \cite{ST}.) An orientable two-manifold (surface) is the sphere or the connected sum of $g$ tori, for $g \geq 1$. 
For any positive integer $n$, a distinct nonorientable surface can be produced by replacing $n$ disks with  M\"obius bands. In particular, replacing one disk with a  M\"obius band produces the real projective plane and replacing two disks produces the Klein bottle. The sphere, the $g$-holed tori, and this sequence of nonorientable surfaces form a complete list of compact, boundaryless two-dimensional manifolds.

\medskip  

\noindent{\bf Example 2.1}.  The real projective plane, ${\Bbb RP}^2$, can be viewed as the union of a M\"obius band and a disc. The correspondent model of the M\"obius band is shown in Fig. 3. Note that this model cannot be embedded to ${\Bbb R}^3$.

\medskip

\begin{figure}
\begin{center}

  \includegraphics[clip,scale=0.9]{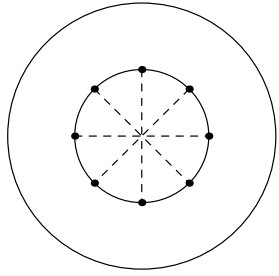}
\end{center}
\caption{M\"obius band.  Diametrically opposite points of the inner boundary circle are to be identified. The outer circle is the  boundary of the M\"obius band. }
\end{figure}

\medskip

Let $T$ be a triangulation of a PL manifold $M$. Then $T$ is a simplicial complex. The vertex set of $T$, denoted by $V(T)$ is the union of the vertex sets of all simplices of $T$. 

Given two triangulations $T_1$ and $T_2$ of two PL manifolds $M_1$ and $M_2$. A {\it simplicial} map is a function $f:V(T_1)\to V(T_2)$ that maps the vertices of $T_1$ to the vertices of $T_2$ and that has the property that for any simplex (face) $s$ of $T_1$, the image set $f(s)$  is a face of $T_2$.

Note that the original Brouwer proof of his fixed point theorem that 
is based on the concept of the degree of a continuous mapping. Let $f:M_1\to M_2$ be a continuous map between two closed manifolds $M_1$ and $M_2$ of the same dimension. Intuitively, the degree  is a number that represents the number of times that the domain manifold wraps around the range manifold under the mapping. Then $\dg2(f)$ (the degree modulo 2) is 1 if this number is odd and 0 otherwise.

It is well known that the degree of a continuous map $f$ of a closed manifold to a manifold is a topological invariant modulo 2 (see, for instance, \cite{Milnor} and \cite[pp. 44--46]{Mat}). Therefore, the degree of $f$ is odd if any  generic point in the range of the map has an odd number of preimages.

Now we define $\dg2(f)$ more rigorously. Let $T_1$ be a triangulation of a closed  $d$-dimensional PL manifold $M_1$. Suppose that $T_2$ is a triangulation of a  $d$-dimensional PL manifold $M_2$. (We do not assume that $M_2$ is closed.) Let $f:V(T_1)\to V(T_2)$ be a simplicial map. 
Consider any $d$-simplex $s$ of $T_2$. Denote by $m$ the number of preimages of $s$ in $T_1$. Then $\dg2(f)=1$ if $m$ is odd and $\dg2(f)=0$ if $m$ is even. Since the parity of $m$ does not depend on $s$, the {\it degree of map modulo 2} is well defined. Thus, the degree of a continuous map of a closed manifold to a manifold is a topological invariant modulo 2. 

  Let $f:M_1\to M_2$ be a continuous map between two manifolds $M_i$ with $d_1:=\dim(M_1)\ge d_2:=\dim(M_2)$.   Then for a point $y\in M_2$ the map $f$ is called {\it transversal to $y$}  (or {\it generic with respect to $y$)} if there are  open sets $U_i\subset M_i$ such that $U_2$ contains $y$,  $U_2=f(U_1)$ and $U_1=f^{-1}(U_2)$. 
 In the case $M_2= {\Bbb R}^{d_2}$ and $y=0$, where $0$ is the origin of ${\Bbb R}^{d_2}$,  $f$ is called {\it transversal to zero}.  

\medskip  

 Let  $M$ be a closed PL-manifold. A simplicial map $A:M\to M$ is called  a {\it free involution} if  $A(A(x))=x$ and $A(x)\ne x$ for all $x\in M$. A triangulation $T$ of $M$ is called {\it antipodal} or {\it equivariant} if $A:T\to T$ is a simplicial map. Let us call a pair $(M,A)$, where $A$ is a free simplicial involution as ${\Bbb Z}_2$-manifold.

\medskip  

\noindent{\bf Example 2.2}. It is clear that $({\Bbb S}^d,A)$ with $A(x)=-x$ is a  ${\Bbb Z}_2$-manifold.   Suppose that $M$ can be represented as a connected sum $N\# N$, where $N$ is a closed PL manifold. Then 
 admits a free involution. Indeed, $M$ can be  ``centrally symmetric" embedded to ${\Bbb R}^k$ with some $k$ and the antipodal symmetry $x\to -x$ in ${\Bbb R}^k$ implies a free involution $T:M\to M$ \cite[Corollary 1]{Mus}. For instance,   orientable  two-dimensional  manifolds $M^2_g$ with even genus $g$ and non-orientable manifolds $P^2_m$ with even $m$, where $m$ is the number of  M\"obius bands,  are ${\Bbb Z}_2$-manifolds.

\medskip

\begin{figure}
\begin{center}

  \includegraphics[clip,scale=0.9]{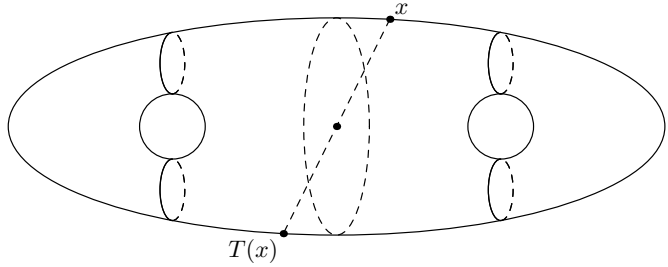}
\end{center}
\caption{The double torus that is centrally symmetric embedded to ${\Bbb R}^3$.  }
\end{figure}

\medskip
 
Suppose that $M$ admits a free simplicial involution $A$. We say that a map $f:M \to {\Bbb R}^d$  is {\it antipodal} (or equivariant) if $f(A(x))=-f(x)$.

\section{Extensions of Sperner's lemma for manifolds}

 A $d$-simplex $S$ where each corner is labelled between 1 and $d+1$ such that all labels are used exactly once is called fully labelled. Suppose that points are added in $S$, then it may be  triangulated, i.e. subdivided into smaller $d$-simplices such that none of the smaller simplices contain any points: all the points are corners of smaller simplices. This subdivision may be done in many ways.

Now label all the interior points according to the following rule: an interior point that is on a facet of the simplex must be given one of the labels of one of the corners of that facet.  The result is called a Sperner labelling.

Note that in this definition an interior point of $S$ can be labelled by any label. So Sperner's constraint is only for the boundary of $S$ that is homeomorphic to ${\Bbb S}^{d-1}$. Let us extend this definition to any closed manifold. 

\medskip 

\noindent{\bf Definition.} 
Let $K$ be a closed  $m$-dimensional PL manifold with vertices $V=\{ v_1,\ldots,v_n\}$ and  faces $\{ F_i\}$ of dimension from 1 to $m$. 
Let $T$ be a triangulation of $K$ such that for any face (that is a simplex) $F_i$ it is a triangulation of $F_i$. Suppose that the vertices of $T$ have a labelling satisfying the following conditions: each vertex $v_k$ of $V$ is assigned a unique label from $\{ 1,2,\ldots,n\}$, and each other vertex $v$ of $T$ belonging to a face $F_i$ with vertices $V(F_i):=\{ v_{i_1},\ldots,v_{i_\ell}\}$ from $V$ is assigned a label of one of the vertices of $V(F_i)$. Such a labelling is called a {\it Sperner labelling} of $T$. 

\medskip 

\noindent{\bf Definition.} We say that a $d$-simplex is a {\it fully labelled} cell or simply a {\it full cell} if all its labels are are distinct. 

Let $T$ be a triangulation of a $d$-dimensional PL manifold. Let $L:V(T)\to\{1,\ldots,n\}$ be a labelling of $T$. Let $Q$ be a set of $(d+1)$-subsets of $\{1,\ldots,n\}$. We denote by $\fc(L,T,Q)$ the number of fully labelled cells that are labelled as labels in $Q$.  
In the case $n=d+1$ we denote $\fc(L,T):=\fc(L,T,\{1,\ldots,d+1\})$.




\medskip

Let $P$ be  a set  of $n$ points $p_1,\ldots,p_n$ in  ${\Bbb R}^d$. Denote by $S(P)$ the collection of all simplices spanned  by vertices $\{p_{i_1},\ldots,p_{i_k}\}$ with $1\le k\le d+1$. Consider a point  $x\in{\Bbb R}^d$  and the set $S_x(P)$ of all simplices from $S(P)$ which cover $x$. If no such simplices exist, we write  $S_x(P)=\emptyset$. 
Denote this set of simplices by $\cov_P(x)$ or just  by $\cov(x)$. 

\medskip

\noindent{\bf Example 3.1}. Let $P$ be a pentagon, see Fig. 5.  Then 
$$ 
\cov(p_1)=(123)\cup(124)\cup(125); \; 
\cov(p_2)=(135)\cup(145)\cup(235)\cup(245); \; 
$$
$$ 
\cov(p_3)=(134)\cup(234)\cup(345); \; 
\cov(O)=(124)\cup(134)\cup(135)\cup(235)\cup(245); \; 
$$

\medskip

\begin{figure}
\begin{center}

  \includegraphics[clip,scale=1.2]{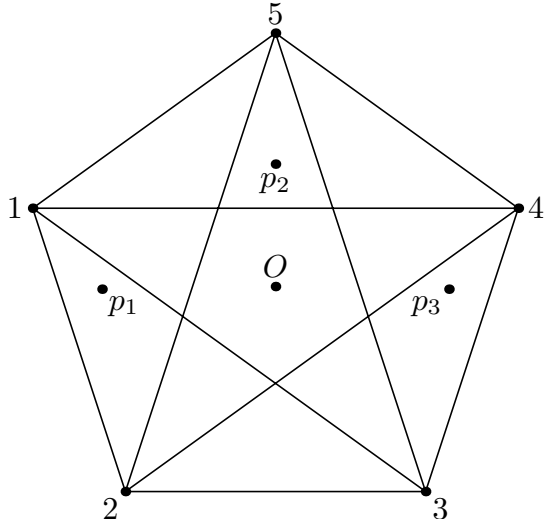}
\end{center}
\caption{Pebbles and $\cov(x)$ for a pentagon}
\end{figure}




\medskip

\noindent{\bf Definition.} Let $P:=\{p_1,\ldots,p_n\}$ be points in ${\Bbb R}^d$. Let  $T$ be a triangulation of a closed PL manifold $M$ of dimension $m$. Let $L$ be an {\it $n$-labelling} of $T$, i.e. a labelling (map) $L:V(T)\to\{1,2,\ldots,n\}$. 
If for  $v\in V(T)$ we have $L(v)=i$, then set  $f_{L,P}(v):=p_i$. Therefore, $f_{L,P}$ is defined for all vertices of $T$, and it uniquely defines a simplicial (piecewise linear) map $f_{L,P}:M\to {\Bbb R}^d$.

\begin{theorem} \label{SpM} Let $P:=\{p_1,\ldots,p_n\}\subset{\Bbb R}^d$. Suppose  $x\in{\Bbb R}^d$ is such that $\cov_P(x)$ consists of $d$-simplices. Let $M$ be a a closed PL $d$-dimensional manifold. Then any  $n$-labelling $L$ of a triangulation $T$ of  $M$ must contain an even number of full cells which are labelled as simplices in $\cov_P(x)$. 
\end{theorem}


\begin{proof} Consider $f_{L,P}:T\to {\Bbb R}^d$. It is easy to see that $\dg2(f_{L,P})=0$. 
Indeed, if $y\in{\Bbb R}^d$ lies outside of the convex hull of $P$ in ${\Bbb R}^d$, then $f^{-1}_{L,P}(y)=\emptyset$. Therefore, for any  point $x$ in ${\Bbb R}^d$ which is a regular value of $f_{L,P}$, we have  $|f^{-1}_{L,P}(x)|\equiv |f^{-1}_{L,P}(y)|\equiv 0 \pmod {2}$. Thus, the number of full cells which are labelled as simplices in $\cov_P(x)$ is even. 
\end{proof}	

For the classical case $n=d+1$ we have the following result (also see \cite{KyFan67}): 
\begin{cor} \label{Sp1} Let $T$ be a triangulation of a closed PL manifold $M^d$. Any $(d+1)$-labelling of $T$ 
must contain an even number of full cells.	
\end{cor}

\begin{cor} \label{cor1} 
Let $M$ be a $d$-dimensional compact PL manifold  with boundary $B$. Let $B$ be PL homeomorphic to the boundary of a $d$-simplex (i.e. $B\cong{\Bbb S}^{d-1}$) with vertices $v_1,\ldots,v_{d+1}$. Then any $(d+1)$-labelling $L$ of a triangulation $T$ of $M$ such that $L(v_i)=i$ and $L$ is a Sperner labelling on the boundary $B$ must contain an odd number of full cells. 
\end{cor}

\begin{figure}[h]
\begin{center}
  \includegraphics[clip]{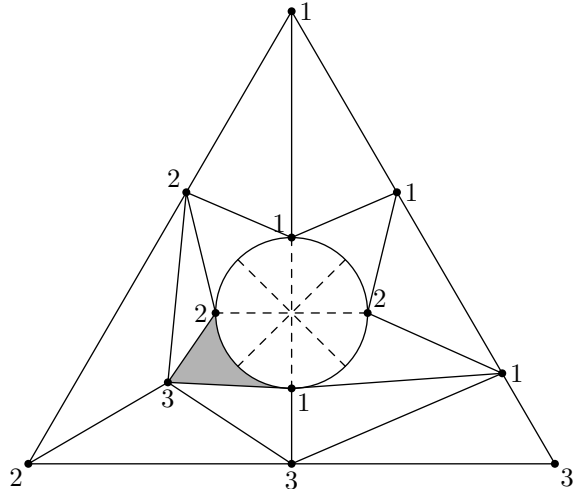}
\end{center}
\caption{Sperner's lemma for   the M\"obius band}
\end{figure}

\begin{proof} We prove this corollary by induction on $d$. It is clear for $d=1$. 
Let $S$ denote a $d$-simplex. Two manifolds $M$ and $S$ can be glued together along $B$. We denote the new manifold by $N$. Then $N$ is a closed manifold. Corollary \ref{Sp1} implies that any $(d+1)$-labelling of any triangulation of $N$ has an even number of full cells. 

  Let us add to the vertices of $T$ one more vertex $q$ that is an internal point of $S$.  Let $C=$cone$(T|_B)$ be the {\it cone triangulation} of $S$ with vertex $q$. (Here $T|_B$ denote the triangulation $T$ on $B$.) Actually, $C$ consists of simplices formed by the  union of all segments connecting the points of $B$ with $q$ and the boundary triangulation $T|_B$. Then we obtain the triangulation $\tilde T:=T\cup C$ of $N$. 

Consider the following labelling $\tilde L$ on $\tilde T$. Let $\tilde L(v):=L(v)$ for all $v\in V(T)$ and $\tilde L(q):=1$.  
Since  $\tilde T$ is a triangulation of $N$, we have that the number of full cells  $\fc(\tilde L,\tilde T)$ is even. 

By induction  the face $F=v_2\ldots v_{d+1}$ \, of $B$ has an odd number of full cells. Then $\fc(\tilde L,C)=\fc(L,T|_F,\{2,\ldots,d+1\})$ is odd. Note that $$\fc(\tilde L,\tilde T)=\fc(L,T)+\fc(\tilde L,C).$$ Thus  $T$ must contain an odd number of full cells. 	
\end{proof}





Note that for the case when $M$ is a $d$-simplex Corollary \ref{cor1} is Sperner's lemma. 

\medskip


Now we show how  the De Loera - Peterson - Su theorem follows from Theorem \ref{SpM}. 

\begin{cor} \label{cor2} Let $P$ be  a convex polytope in ${\Bbb R}^d$ with vertices $p_1,\ldots,p_n$.
Let $M$ be a compact $d$-dimensional PL manifold  with boundary $B$. Let $B$ be piecewise linearly  homeomorphic to the boundary of $P$. Suppose  $x\in{\Bbb R}^d$ is such that $\cov_P(x)$ consists of $d$-simplices. Then any $n$-labelling $L$ of a triangulation $T$ of $M$ that is a Sperner labelling on the boundary must contain an odd number of full cells which are labelled as simplices in $\cov_P(x)$. In other words, $\fc(L,T,\cov(x))$ is odd. 
\end{cor}
\begin{proof} This corollary can be proved by similar arguments as Corollary \ref{cor1}. Indeed, two manifolds $M$ and $P$ can be glued together along $B$. We denote the new manifold by $N$. Then $N$ is a closed manifold.
	
Let $C:=$cone$(T|_B)$ be the cone triangulation of $P$ with vertex $q$, where $q$ is an internal point of $P$. Then we have the triangulation $\tilde T:=T\cup C$ of $N$. 

	Consider the following labelling $\tilde L$ on $\tilde T$. Let $\tilde L(v):=L(v)$ for all $v\in V(T)$ and $\tilde L(q):=1$.   Now we show that $\fc(\tilde L,C,\cov(x))$ is odd.
	
Consider the line in ${\Bbb R}^d$ passes through points $p_1$ and $x$. By assumptions, this line intersects the boundary $B$ of the polytope $P$ in two points $p_1$ and $y$, where $y$ is an internal point of some $(d-1)$-simplex of $T|_B$ with  distinct labels $\ell_1,\ldots,\ell_d$. Therefore, $y$ lies on the face $F=v_{\ell_1}\ldots v_{\ell_d}$.   By induction  $\fc(L,T|_F,\{\ell_1,\ldots,\ell_d)$ is odd. Note that $\tilde L$ on $C$ contains only one labelling $(1\ell_1\ldots\ell_d)$ from $\cov(x)$. Then  
$$\fc(\tilde L,C,\cov(x))=\fc(\tilde L,C,\{1,\ell_1,\ldots,\ell_d\})=\fc(L,T|_F,\{\ell_1,\ldots,\ell_d\})=1 \md 2).$$ 

We have $$\fc(\tilde L,\tilde T,\cov(x))=\fc(L,T,\cov(x))+\fc(\tilde L,C,\cov(x)),$$ where $\fc(\tilde L,\tilde T,\cov(x))$ is even and $\fc(\tilde L,C,\cov(x))$ is odd. Thus $\fc(L,T,\cov(x))$ is odd. 
\end{proof}

\begin{cor} Let $P$ be  a convex polytope in ${\Bbb R}^d$ with vertices $p_1,\ldots,p_n$.
Let $M$ be a compact $d$-dimensional PL manifold  with boundary $B$. Let $B$ be PL  homeomorphic to the boundary of $P$.  Then any $n$-labelling of a triangulation of $M$ that is a Sperner labelling on the boundary contains at least $n-d$  full cells.
\end{cor}

For the case $M=P$ this statement is the polytopal Sperner lemma \cite[Th. 1]{DeLPS}. 

\begin{proof} A proof of this corollary follows from  another theorem from  \cite[Th. 4]{DeLPS}: {\it Any convex polytope $P$ in ${\Bbb R}^d$ with $n$ vertices contains a pebble set of size $n-d$.}  (A finite set of points ({\it pebbles}) in $P$ is called a {\it pebble set} if each $d$-simplex of $P$ contains at most one pebble interior to chambers.) 
	
Consider a pebble set $\{p_i\}$ of size $n-d$. Then for $i\ne j$ we have $\cov(p_i)\cap\cov(p_j)=\emptyset$. Thus Corollary \ref{cor2} guarantees that there are at least $n-d$ full cells.  	  	
\end{proof}

\noindent{\bf Remark.} In fact, Meunier's proof of his extension of the polytopal Sperner lemma (De Loera - Peterson - Su's theorem) is not based on the ``pebbles set theorem.'' It is an interesting problem to find an extension of Meunier's theorem for manifolds. 

\medskip 

\noindent{\bf Example 3.2}. Consider the case when $P$ is a pentagon. In $P$ there are three pebbles $p_1,p_2,p_3$, see Fig. 5. So the polytopal Sperner's lemma (Theorem 1.2) and Corollary 3.4 implies that there are at least three fully labelled triangles. Actually, this statement can be improved. 

Note that there are 10 5-labellings for triangles. Five of them are consecutive: (123), (234), (345), (451), (512) and five are non-consecutive. In fact, $\cov(O)$ consists of non-consecutive labellings,  see Example 3.1.   Then Corollary 3.3 implies the following statement:\\  
{\it 
Any Sperner 5-labelling of a triangulation $T$ of a pentagon $P$ must contain at least three full cells. Moreover, at least one of them is not consecutive labelled.  
}

\section{Extensions of Tucker's lemma for manifolds}


\medskip

\noindent{\bf Definition. } Let  $M$ be a closed PL $d$-dimensional manifold  with a free simplicial involution $A:M\to M$.
We say that a pair $(M,A)$ is a {\it BUT (Borsuk-Ulam Type) manifold} if for any continuous  $g:M \to {\Bbb R}^d$ there is a point $x\in M$ such that $g(A(x))=g(x)$. Equivalently, if a continuous  map $f:M \to {\Bbb R}^d$  is { antipodal},  then the zeros set $Z_f:=f^{-1}(0)$ is not empty.

\medskip

In \cite{Mus}, we found several equivalent necessary and sufficient conditions for manifolds to be BUT. For instance, {\it $M$ is a BUT manifold if and only if $M$ admits an antipodal continuous transversal to zeros map $h:M \to {\Bbb R}^d$ with $|Z_h|=2\pmod{4}$.} 

Let $T$ be any equivariant triangulation  of $M$. We say that $$L:V(T)\to \{+1,-1,+2,-2,\ldots, +d,-d\}$$ is an {\it equivariant} (or {\it Tucker's}) labelling if    $L(A(v))=-L(v)$).

An edge $e$ in  $T$ is called {\it complementary} if its two ends are labelled by opposite numbers, i.e. if $e=uv$, then $L(v)=-L(u)$. 






\begin{theorem} \label{TBUT}  A closed PL $d$-dimensional manifold  $M$  with a free simplicial involution $A$ is  BUT  if and only if for any equivariant labelling of any equivariant triangulation $T$ of $M$ there exists a complementary edge.
\end{theorem}

For the case $M={\Bbb S}^d$ this is Tucker's lemma. 

\begin{proof} Let  $e_1,\ldots, e_d$ be an orthonormal basis of ${\Bbb R}^d$.  Any equivariant labelling $L$ of a triangulation $T$ of $M$  defines a simplicial map $f_L:T\to C^d$, where $C^d$ is the crosspolytope in ${\Bbb R}^d$  with vertex set $\{e_1,-e_1,e_2,-e_2,\ldots,e_d,-e_d\}$,  
where for $v\in V(T)$, $f_L(v)=e_i$ if $L(v)=i$ and  $f_L(v)=-e_i$ if $L(v)=-i$. (See details in \cite[Sec. 2.3]{Mat}.) In other words, $f_L=f_{L,C^d}$ (see Section 3). 

Note that any fully labelled simplex contains a complementary edge. Therefore, if $L$ has no complementary edges, then $f_L:T\to {\Bbb R}^d$  has no zeros.

The reverse implication can be proved by the same arguments as equivalence of the Borsuk-Ulam theorem and Tucker's lemma in \cite[2.3.2]{Mat}, i.e. if there is continuous antipodal map $f:M\to{\Bbb S}^{d-1}$ (i.e. $Z_f=\emptyset$) then $T$ and $L$  can be  constructed with no complementary edges. (See also Theorem \ref{PolTuck}.)   
\end{proof}	

Theorem \ref{TBUT} and \cite[Theorem 2]{Mus} immediately imply: 

\begin{cor} Let $M$ be a closed PL manifold with a free involution $A$. Then $M$ is a BUT manifold if and only if  there exist an equivariant triangulation $\Lambda$ of $M$ and an equivariant labelling of $V(\Lambda)$  such that $f_L:\Lambda\to {\Bbb R}^d$ is transversal to zeros and the number of complementary edges is  $4k+2$, where $k$ is integer.
\end{cor}

\begin{cor} \label{TuckM} Let $T$ be a triangulation of a PL-compact $d$-dimensional manifold  $M$ with  boundary $B$ that is homeomorphic to ${\Bbb S}^{d-1}$. Assume $T$ is antipodally symmetric on the boundary. 
	Let $L:V(T)\to \{+1,-1,+2,-2,\ldots,+d,-d\}$ be a labelling of the vertices of $T$ which satisfies $L(-v)=-L(v)$ for all vertices $v$ in  $B$. Then there is a complementary edge in $T$. 
\end{cor}
\begin{proof}
Consider two copies of $M$: $M_+$ and $M_-$, where for $M_+$ we take a given labelling $L$ and for $M_-$ we take a labelling $\bar L=(-L)$, i.e. $\bar L(v)=-L(v)$. Since $L$ is antipodal on the boundary $B={\Bbb S}^{d-1}$ the connected sum $N:=M\# M$ with a free involution $I:N\to N$, where $I(M_+)=M_-$, is well defined.  \cite[Corollary 1]{Mus} implies that $N$ is BUT. Thus, from Theorem \ref{TBUT} follows that there is a complementary edge. 
\end{proof}	

\begin{figure}
\begin{center}

  \includegraphics[clip,scale=0.9]{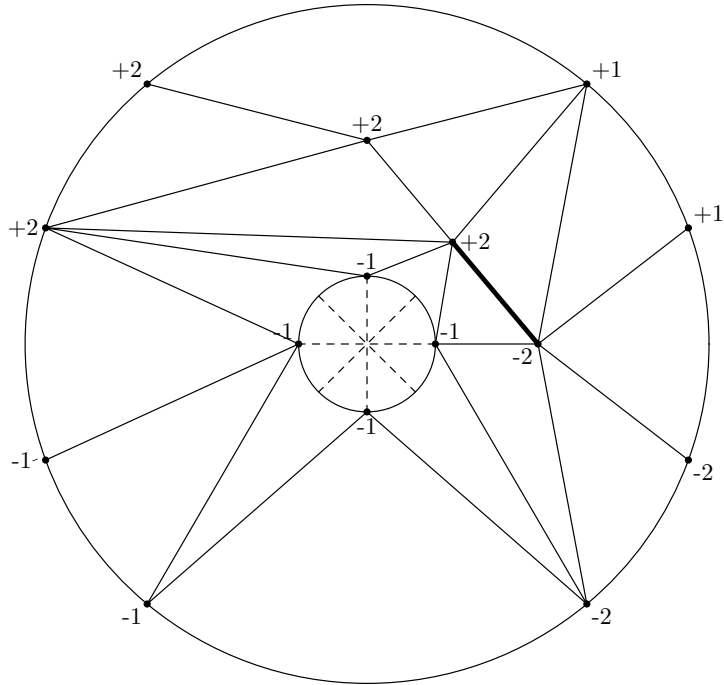}
\end{center}
\caption{Tucker's lemma for   the M\"obius band}
\end{figure}

Now we extend Theorem \ref{TBUT} for $n$-labellings. 

\begin{theorem} \label{PolTuck} 
Let  $P=\{p_1,-p_1,\ldots,p_n,-p_n\}$ be a centrally symmetric set of $2n$ points in ${\Bbb R}^d$. Let points in $P$ be equivariantly labelled by  $\{+1,-1,+2,-2,\ldots,+n,-n\}$. Let $M$ be a closed PL $d$-dimensional manifold with a free involution.  Then $M$ is a BUT manifold if and only if for any equivariant triangulation $T$  of $M$ and for any  equivariant  labelling  $L:V(T)\to \{+1,-1,+2,-2,\ldots, +n,-n\}$ there exists a simplex $s$ in $T$ such that $0\in f_{L,P}(s)$. 
\end{theorem}
\begin{proof} If $M$ is BUT, then $f_{L,P}$ has zeros, so there is a simplex $s$ as required. 
	
\medskip
	
Suppose $M$ is not BUT. Then there is a continuous antipodal $h:M\to {\Bbb S}^{d-1}$. Let $T$ be an equivariant triangulation of $M$. Let $Q$ denote the boundary of the convex hull of $P$ in ${\Bbb R}^d$. Without loss of generality we may assume that $h:M\to Q$ and for any vertex $v\in V(T)$ the image $f(v)$ has only one closest vertex $p$ in $Q$. Then set $L(v):=L(p)$. 
This labelling implies that $f_{L,P}$ is an antipodal simplicial map from $T$ to $Q$. Thus, $0$ in ${\Bbb R}^d$ is not covered by $f_{L,P}$, a contradiction. 
\end{proof}		

\section{Radon partitions and Ky Fan's lemma for manifolds}

In this section we show that Ky Fan's lemma  follows from Theorem \ref{PolTuck}. 

Radon's theorem on convex sets states that {\it any set $S$ of $d + 2$ points in ${\Bbb R}^d$ can be partitioned into two (disjoint) sets $A$ and $B$ whose convex hulls intersect. Moreover, if rank$(S)=d$, then this partition is unique.} 

The  partition $S=A\bigcup B$ is called the {\it Radon partition} of $S$. 

Breen \cite{Breen} proved that if $S$ is a $(d+2)$-subset of the moment curve $C_d$ in ${\Bbb R}^d$, then $S=A\bigcup B$ is the Radon partition if and only if $A$ and $B$ alternate along $C_d$. Actually, Breen's theorem can be extended for convex curves in ${\Bbb R}^d$. 

We say that a curve $K$ in $\mathbb{R}^d$ is {\it convex} if for every hyperplane $K$ intersects it  at no more than $d$ points. It is well known that the moment curve $C_d$ is convex. In \cite[Sec. 3]{Mus98} we considered several other examples of convex curves. 

\medskip

\noindent{\bf Definition.}
Let $K=\{x(t)=(x_1(t),\ldots,x_d(t)): t\in[a,b]\}$ be  a curve in $\mathbb{R}^d$. Let $S=\{x(t_1),\ldots,x(t_{d+2})\}$, where $a<t_1<t_2<\ldots<t_{d+2}< b$.  We say that $A$ and $B$  {\it alternate along} $K$ if $S=A\bigcup B$, where  $A=\{x(t_1),x(t_3),\ldots\}$ and $B=\{x(t_2),x(t_4),\ldots\}$.  

\begin{theorem}\label{convcurve} A curve $K$ in $\mathbb{R}^d$ is convex if and only if for any $(d+2)$-subset $S$ of $K$ its Radon partition sets $A$ and $B$ alternate along $K$. 
\end{theorem}
\begin{proof} Let $K$ be convex and $S=\{x(t_1),\ldots,x(t_{d+2})\}$ be a $(d+2)$-subset of $K$. Let $A\bigcap B$ be the Radon partition of $S$. If $A$ and $B$ do not alternate along $K$, there are at most $d$ points $P=\{x(\tau_i)\}$ which separate $A$ and $B$ on $K$. If $r=|P|<d$, we add to $P$ $d-r$ points $x(\tau)$  with $\tau\in(a,t_1)$. Then $P$ defines a hyperplane $H$ which passes through the points in $P$.  Clearly, $H$ separates $A$ and $B$  in $\mathbb{R}^d$. Thus, $A\bigcap B$ cannot be the Radon partition of $S$, a contradiction.

Suppose that for any $(d+2)$-subset $S$ of $K$, its Radon partition sets $A$ and $B$ alternate along $K$. If $K$ is not convex, then there is a hyperplane $H$ which intersects $K$ at $r\ge d+1$ points. Therefore, $H$ separates $K$ into $r+1$ connected components $C_1,\ldots,C_{r+1}$. Let $S=\{x(t_1),\ldots,x(t_{d+2})\}$, where $x(t_i)\in C_i$. Since $A$ and $B$ which alternate along $K$ are separated by $H$, the partition $S=A\bigcup B$ is not Radon's - a contradiction.
\end{proof}

\medskip

\noindent{\bf Definition.}  Let $P$ be a convex polytope in $\mathbb{R}^d$ with $2n$ centrally symmetric vertices $\{p_1,-p_1,\ldots,p_n,-p_n\}$.  We say that $P$  is {\it ACS (Alternating Centrally Symmetric)} $(n,d)$-polytope if the set of all simplices in $\cov_P(0)$, that contain the origin $0$ of $\mathbb{R}^d$ inside,  consists of  edges $(p_i,-p_i)$ and  $d$-simplices with vertices \{$p_{k_0},-p_{k_1},\ldots,(-1)^dp_{k_d}$\} and \{$-p_{k_0},p_{k_1},\ldots,(-1)^{d+1}p_{k_d}$\}, where $1\le k_0<k_1<\ldots<k_d\le n$.  

\begin{theorem} \label{corRP} For any integer $d\ge 2$ and $n\ge d$ there exists ASC $(n,d)$-polytope. 
\end{theorem}
\begin{proof} Let $q_1,\ldots,q_n$ be points on a convex curve $K$ in $\mathbb{R}^{d-1}$. Let $p_i=(q_i,1)\in \mathbb{R}^d$. Denote by $P(n,d)$ a convex polytope with vertices $\{p_1,-p_1,\ldots,p_n,-p_n\}$.  Clearly, $0\in(-p_i,p_i)$. Let $\Delta$ be a simplex spaned by vertices of $P(n,d)$.  Let $V(\Delta)=A\bigcup(-B)$, where  $A$ and $B$ are vertices with $x_d=1$. It is easy to see that  $0\in \Delta$ if and only if $\conv(A)\bigcap\conv(B)=\emptyset$, i.e. $S=A\bigcup B$ is the {Radon partition} of $S$. Then Theorem \ref{convcurve} implies that $A$ and $B$ alternate along $K$. Thus, $P(n,d)$ is an ASC $(n,d)$-polytope. 
\end{proof}

Let $P$ be an ASC $(n,d)$-polytope. If we apply Theorem \ref{PolTuck} for $P$, then we obtain the following theorem. 

\begin{theorem} Let $M$ be a  BUT $d$-dimensional manifold with a free involution $A$. Let $T$ be any equivariant triangulation  of $M$. Let $L:V(T)\to \{+1,-1,+2,-2,\ldots, +n,-n\}$ be an equivariant  labelling. Suppose that there are no  complementary edges in $T$. Then there are an odd number of $d$-simplices with labels in the form $\{k_0,-k_1,k_2,\ldots,(-1)^dk_d\}$, where  $1\le k_0<k_1<\ldots<k_d\le n$.
\end{theorem}

For the case $M={\Bbb S}^d$ this theorem is Ky Fan's combinatorial lemma \cite{KyFan}. Actually, it is a new proof of this lemma.



\section{Sperner and Tucker's type lemmas for the case $m\ge d$}







Now we consider extensions of the polytopal Sperner and Tucker lemmas for the case when $d\le\dim{M}=m$. In this case, the set of fully-colored $d$-simplices defines certain $(m-d)$-submanifold  $S$ of $M$. A natural extension of Theorem \ref{SpM} is that $S$ is cobordant to zero. We also consider an extension of the Tucker lemma.

An $m$-dimensional manifold $M$ is called {\it null-cobordant} (or  {\it cobordant to zero}) if there is a cobordism between $M$ and the empty manifold; in other words, if $M$ is the entire boundary of some $(m+1)$-manifold. Equivalently, its cobordism class is trivial.

\begin{theorem} \label{BordiS} Let $P=\{p_1,\ldots,p_n\}$ be a set of points in ${\Bbb R}^d$. Suppose  $y\in{\Bbb R}^d$ is such that $\cov_P(y)$ consists of $d$-simplices. Let $M$ be a a closed PL $m$-dimensional manifold with $m\ge d$. Then for any  $n$-labelling $L$ of a triangulation $T$ of $M$ the set $S:=f_{L,P}^{-1}(y)$ is a null-cobordant manifold of dimension $d-m$. 	
\end{theorem}

Note that for $d=m$ this theorem yields Theorem \ref{SpM}. In this case $S$ consists of even number of points. 

\begin{proof} Let 
	$$W:=M\times[0,1], \; M_0:=M\times \{0\} \mbox{ and } M_1:=M\times \{1\}.$$ 
	Let $f_0:=f_{L,P}:M_0\to{\Bbb R}^d$. Let us fix a point $q\ne y$  in ${\Bbb R}^d$  and set $f_1(x)=q$ for all $x\in M_1$.

	Note that $f_0$ is transversal to $y$ and $f_1^{-1}(y)$ is empty. 
	Let 
$$
 F(x,t):=(1-t)f_0(x)+tf_1(x)
$$	
Then  $F: W\to {\Bbb R}^d$ is  transversal to $y$ with $F|_{M_0}=f_0$ and $F|_{M_1}=f_1$. 
Therefore, $Z_F:=F^{-1}(y)$ is a manifold of dimension $(m+1-d)$.
	
Denote $Z_{i}:=Z_F\bigcap M_i=f_i^{-1}(y), \, i=0,1.$ It is clear that $Z_0=S$ and  $Z_1$ is empty. Thus, $Z_0$ is the boundary of $Z_F$ and so it is a null-cobordant $(m-d)$-dimensional manifold.  	
\end{proof}

Now we extend the class of BUT manifolds. 

\medskip

\noindent{\bf Definition. } 
We say that a closed PL-free $m$-dimensional ${\Bbb Z}_2$-manifold $(M,A)$ is a $\but_{m,d}$  if for any continuous  $g:M \to {\Bbb R}^d$ there is a point $x\in M$ such that $g(A(x))=g(x)$. Equivalently, if a continuous  map $f:M \to {\Bbb R}^d$  is  antipodal,  then the zeros set $Z_f:=f^{-1}(0)$ is not empty.

\medskip

We obviously have 
$$ 
\but=\but_{m,m}\subset \but_{m,m-1}\subset \ldots \subset \but_{m,1}. 
$$
Note that in our paper \cite{Mus} we found a sufficient condition for $(M,\Lambda)$ to be a BUT$_{m,d}$, see \cite[Corollary 3]{Mus}.

Let $T$ be an antipodal triangulation of $M$ Any equivariant labelling  $L:V(T)\to \{+1,-1,+2,-2,\ldots, +d,-d\}$  defines a simplicial map $f_L:T\to C^d$, where $C^d$ is the crosspolytope in ${\Bbb R}^d$ (see Section 4). 
It is easy to see that if $L$ has no complementary edges, then $f_L:T\to {\Bbb R}^d$  has no zeros. It implies the following theorem.

\begin{theorem} Let $m\ge d$. Let $T$ be any equivariant triangulation  of a $\but_{m,d}$ manifold  $(M,A)$. Let $L:V(A)\to \{+1,-1,+2,-2,\ldots, +d,-d\}$  be any equivariant labelling of $T$. Then there exists a complementary edge in $T$.
\end{theorem}

This theorem is an extension of Theorem 4.1. When $m\ge d$, it is not hard  to extend other theorems and corollaries from Sections 4 and 5. 

\medskip

\medskip
  
\medskip

\noindent{\bf Acknowledgment.} I  wish to thank Arseniy Akopyan and Fr\'ed\'eric Meunier  for helpful discussions and  comments. 

 \medskip

\noindent O. R. Musin\\ 
Department of Mathematics, University of Texas at Brownsville, One West University Boulevard, Brownsville, TX, 78520 \\
and\\
IITP RAS, Bolshoy Karetny per. 19, Moscow, 127994, Russia\\ 
{\it E-mail address:} oleg.musin@utb.edu

\end{document}